\documentclass[11pt]{article}

\usepackage{amsfonts,amsmath,latexsym,color,epsfig}
\newtheorem{theorem}{Theorem}
\newtheorem{lemma}{Lemma}

\newenvironment{proof}
      {\medskip\noindent{\bf Proof:}\hspace{1mm}}
      {\hfill$\Box$\medskip}

\input{epsf}

\makeatletter
\def\Ddots{\mathinner{\mkern1mu\raise\p@
\vbox{\kern7\p@\hbox{.}}\mkern2mu
\raise4\p@\hbox{.}\mkern2mu\raise7\p@\hbox{.}\mkern1mu}}
\makeatother

\title{\vspace{-0.7cm}Sidorenko's conjecture for a class of graphs: an exposition}
\author{David Conlon\thanks{Mathematical Institute, Oxford OX1 3LB, UK. Email: {\tt david.conlon@maths.ox.ac.uk}.} \and Jacob Fox\thanks{Department of Mathematics, MIT, Cambridge, MA 02139-4307. Email: {\tt fox@math.mit.edu}.} \and Benny Sudakov\thanks{Department of Mathematics, UCLA,  Los Angeles, CA 90095. Email: {\tt bsudakov@math.ucla.edu}.}}
\date{}
\begin{document}
\maketitle

A famous conjecture of Sidorenko \cite{Si3} and Erd\H{o}s-Simonovits \cite{Sim} states that if $H$ is a bipartite graph then the random graph with edge density $p$ has in expectation asymptotically the minimum number of copies of $H$ over all graphs of the same order and edge density. The goal of this expository note is to give a short self-contained 
proof (suitable for teaching in class) of the conjecture if $H$ has a vertex complete to all vertices in the other part. This was originally proved in \cite{CFS10}.

\begin{theorem}\label{main}
Sidorenko's conjecture holds for every bipartite graph $H$ which has a vertex complete to the other part.
\end{theorem}

The original formulation of the conjecture by Sidorenko is in terms of graph homomorphisms. A homomorphism from a graph $H$ to a graph $G$ is a mapping $f:V(H) \rightarrow V(G)$ such that for each edge $(u,v)$ of $H$, $(f(u),f(v))$ is an edge of $G$. Let $h_H(G)$ denote the number of homomorphisms from $H$ to $G$. We also consider the normalized function $t_H(G)=h_H(G)/|G|^{|H|}$, which is the fraction of mappings $f:V(H) \rightarrow V(G)$ which are homomorphisms. Sidorenko's conjecture states that for every bipartite graph $H$ with $m$ edges and every graph $G$, $t_H(G) \geq t_{K_2}(G)^m$. We will prove that this is the case for $H$ as in Theorem~\ref{main}.

We use a probabilistic technique known as dependent random choice. The idea is that most small subsets of the neighborhood of a random vertex have large common neighborhood. Our first lemma gives a counting version of this technique. We will then combine this with a simple embedding lemma to give a lower bound for $t_H(G)$ in terms of $t_{K_2}(G)$. For a vertex $v$ in a graph $G$, the {\it neighborhood} $N(v)$ is the set of vertices adjacent to $v$. For a sequence $S$ of vertices of a graph $G$, the {\it common neighborhood} $N(S)$ is the set of vertices adjacent to every vertex in $S$. 

\begin{lemma} \label{goodstep}
Let $G$ be a graph with $N$ vertices and $pN^2/2$ edges. Call a vertex $v$ bad with respect to $k$ if the number of sequences of $k$ vertices in $N(v)$ with at most $(2n)^{-n-1}p^kN$ common neighbors is at least $\frac{1}{2n}|N(v)|^k$. Call $v$ good if it is not bad with respect to $k$ for all $1 \leq k \leq n$. Then the sum of the degrees of the good vertices is at least $pN^2/2$.
\end{lemma}
\begin{proof}
We write $v \sim k$ to denote that $v$ is bad with respect to $k$. Let $X_k$ denote the number of pairs $(v,S)$ with $S$ a sequence of $k$ vertices, $v$ a vertex adjacent to every vertex in $S$, and $|N(S)| \leq (2n)^{-n-1}p^kN$. We have $$(2n)^{-n-1}p^k N \cdot N^k \geq X_k \geq \sum_{v, v \sim k}\frac{1}{2n} |N(v)|^k \geq \frac{1}{2n} N(\sum_{v, v \sim k}|N(v)|/N)^k = \frac{1}{2n} N^{1-k}(\sum_{v, v \sim k}|N(v)|)^k.$$
The first inequality is by summing over $S$, the second inequality is by summing over vertices $v$ which are bad with respect to $k$, and the third  inequality is by convexity of the function $f(x)=x^k$. We therefore get $$\sum_{v,v\sim k} |N(v)| \leq (2n)^{-n/k}pN^2 \leq \frac{1}{2n}pN^2.$$
Hence, $$\sum_{v,v~\textrm{good}} |N(v)| \geq \sum_{v} |N(v)| - \sum_{k=1}^n \sum_{v,v\sim k} |N(v)| \geq pN^2-n \cdot \frac{1}{2n}pN^2 = pN^2/2,$$
as required.
\end{proof}

\begin{lemma} \label{randomembed}
Suppose $\mathcal{H}$ is a hypergraph with $v$ vertices and at most $e$ edges and $\mathcal{G}$ is a hypergraph on $N$ vertices with the property that for each $k$, $1 \leq k \leq v$, the number of sequences of $k$ vertices of $\mathcal{G}$ that do not form an edge of $\mathcal{G}$ is at most $\frac{1}{2e}N^k$. Then the number of homomorphisms from $\mathcal{H}$ to $\mathcal{G}$ is at least $\frac{1}{2}N^v$.
\end{lemma}
\begin{proof}
Consider a random mapping of the vertices of $\mathcal{H}$ to the vertices of $\mathcal{G}$. The probability that a given edge of $\mathcal{H}$ does not map to an edge of $\mathcal{G}$ is at most $\frac{1}{2e}$. By the union bound, the probability that there is an edge of $\mathcal{H}$ that does not map to an edge of $\mathcal{G}$ is at most $e \cdot \frac{1}{2e}=1/2$. Hence, with probability at least $1/2$, a random mapping gives a homomorphism, so there are at least $\frac{1}{2}N^v$ homomorphisms from $\mathcal{H}$ to $\mathcal{G}$.
\end{proof}

\begin{lemma}\label{importantstep}
Let $H=(V_1,V_2,E)$ be a bipartite graph with $n$ vertices and $m$ edges such that there is
a vertex $u \in V_1$ which is adjacent to all vertices in $V_2$. Let $G$ be a graph with $N$ vertices and $pN^2/2$ edges, so $t_{K_2}(G)=p$. Then the number of homomorphisms from $H$ to $G$ is at least
$(2n)^{-n^2}p^mN^n$.
\end{lemma}
\begin{proof}
Let $n_i=|V_i|$ for $i \in \{1,2\}$. We will give a lower bound on the number of homomorphisms $f:V(H) \rightarrow V(G)$ that map $u$ to a good vertex $v$ of $G$. Suppose we have already picked $f(u)=v$. Let $\mathcal{H}$ be the hypergraph with vertex set $V_2$, where $S \subset V_2$ is an edge of $\mathcal{H}$ if there is a vertex $w \in V_1 \setminus \{v\}$ such that $N(w)=S$. The number of vertices of $\mathcal{H}$ is $n_2$, which is at most $n$, and the number of edges of $\mathcal{H}$ is $n_1 - 1$, which is also at most $n$. Let $\mathcal{G}$ be the hypergraph on $N(v)$, where a sequence $R$ of $k$ vertices of $N(v)$ is an edge of $\mathcal{G}$ if $|N(R)| \geq (2n)^{-n-1}p^kN$. Since $v$ is good, for each $k$, $1 \leq k \leq v$, the number of sequences of $k$ vertices of $\mathcal{G}$ that are not the vertices of an edge of $\mathcal{G}$ is at most $\frac{1}{2n}N^k$. Hence, by Lemma \ref{randomembed}, there are at least $\frac{1}{2}|N(V)|^{n_2}$ homomorphisms $g$ from $\mathcal{H}$ to $\mathcal{G}$. Pick one such homomorphism $g$, and let $f(x)=g(x)$ for $x \in V_2$. By construction, once we have picked $f(u)$ and $f(V_2)$, there are at least $(2n)^{-n-1}p^{|N(w)|}N$ possible choices for $f(w)$ for each vertex $w \in V_1$. Hence, the number of homomorphisms from $H$ to $G$ is at least
\begin{eqnarray*}\sum_{v~\textrm{good}}\frac{1}{2}|N(v)|^{n_2}\prod_{w \in V_1 \setminus \{v\}}(2n)^{-n-1}p^{|N(w)|}N & = &   \frac{1}{2}(2n)^{-(n-1)(n_1-1)}p^{m-n_2}N^{n_1-1}\sum_{v~\textrm{good}}|N(v)|^{n_2} \\ & \geq &  \frac{1}{2}(2n)^{-(n-1)(n_1-1)}p^{m-n_2}N^{n_1-1}N\left(\sum_{v~\textrm{good}}|N(v)|/N\right)^{n_2} \\ & \geq  & \frac{1}{2}(2n)^{-(n-1)(n_1-1)}p^{m-n_2}N^{n_1}(pN/2)^{n_2}
\\ & \geq & (2n)^{-n^2}p^mN^{n}.
\end{eqnarray*}
The first inequality is by convexity of the function $q(x)=x^k$ and the second inequality is by the lower bound on the sum of the degrees of good vertices given by Lemma \ref{goodstep}.
\end{proof}

We next complete the proof of Theorem \ref{main} by improving the bound in the previous lemma on the number of homomorphisms from $H$ to $G$ using a tensor power trick. The tensor product $F \times G$ of two graphs $F$ and $G$ has vertex set $V(F) \times V(G)$ and any two
vertices $(u_1,u_2)$ and $(v_1,v_2)$ are adjacent in $F \times G$ if and only if $u_i$ is adjacent with $v_i$ for $i \in \{1,2\}$.
Let $G^1=G$ and $G^r=G^{r-1} \times G$. Note that $t_H(F \times G)=t_H(F) \times t_H(G)$ for all $H,F,G$.

\vspace{5mm}
{\bf Proof of Theorem \ref{main}:} Suppose for contradiction that there is a graph $G$ such that $t_{H}(G)<t_{K_2}(G)^m$. Denote the number of edges of $G$ as $pN^2/2$, so $t_{K_2}(G)=p$. Let $c=\frac{t_H(G)}{t_{K_2}(G)^m}<1$. Let $r$ be such that $c^r<(2n)^{-n^2}$. Then
$$t_{H}(G^r)=t_H(G)^r=c^rt_{K_2}(G)^{mr}=c^rt_{K_2}(G^r)^{m}<(2n)^{-n^2}t_{K_2}(G^r)^{m}.$$
However, this contradicts Lemma \ref{importantstep} applied to $H$ and $G^r$. This completes the proof. 
 {\hfill$\Box$\medskip}


\begin{thebibliography}{}

\bibitem{CFS10}
{D. Conlon, J. Fox and B. Sudakov,} {An approximate version of Sidorenko's conjecture,} {\it Geom. Funct. Anal.} {\bf 20} (2010), 1354--1366.

\bibitem{Si3}
A. F. Sidorenko,
\newblock A correlation inequality for bipartite graphs,
\newblock {\it Graphs Combin.} {\bf 9} (1993), 201--204.

\bibitem{Sim}
M. Simonovits, Extremal graph problems, degenerate extremal problems
and super-saturated graphs, in: {\it Progress in graph theory} (J.
A. Bondy ed.), Academic, New York, 1984, 419--437.

\end{thebibliography}
\end{document}